\documentclass[pdftex]{article}
\usepackage{amsmath, amssymb, amsthm, delarray}

\usepackage[top =20mm, left=35mm,right=37mm]{geometry}

\usepackage[pdftex]{graphicx}
\usepackage{tikz}
\usetikzlibrary{intersections,calc,arrows.meta}

\newtheorem{thm}{Theorem}[section]
\newtheorem*{theorem*}{Theorem}
\newtheorem{lem}[thm]{Lemma}
\newtheorem{prop}[thm]{Proposition}
\newtheorem{cor}[thm]{Corollary}
\theoremstyle{definition}

\newtheorem{rem}[thm]{Remark}

\numberwithin{equation}{section}
\theoremstyle{remark}

\begin{document}

\title{\textbf{Perfectly packing a square by squares of sidelength $f(n)^{-t}$}}

\author{Keiju Sono}

\date{}
\allowdisplaybreaks

\maketitle 
\noindent
\begin{abstract}
In this paper, we prove that  for any $1/2<t<1$, there exists a positive integer $N_{0}$ depending on $t$ such that for any $n_{0}\geq N_{0}$,  squares of sidelength $f(n)^{-t}$ for $n\geq n_{0}$ can be  packed with disjoint interiors into  a square of area $\sum _{n=n_{0}}^{\infty}f(n)^{-2t}$,  if the function $f$ satisfies some suitable conditions. The main theorem (Theorem 1.1) is a generalization of Tao's theorem in \cite{T}, which argued the case $f(n)=n$.  As corollaries, we prove that there are such packings of squares when $f(n)$ represents  the $n$th element of either an arithmetic progression or the set of prime numbers.  In these cases, we give effective lower bounds for $N_{0}$ with respect to $t$. Furthermore, we consider the case that $f(n)$ represents the $n$th element of the set of twin primes and prove that squares of sidelength $f(n)^{-t}$ for $n\geq n_{0}$ can be packed with disjoint interiors into a  slightly larger square than theoretically expected. 

\footnote[0]{2020 {\it Mathematics Subject Classification}.  52C15 }
\footnote[0]{{\it Key Words and Phrases}. square packing in a square, perfectly packing,  Meir-Moser conjecture}
\end{abstract}


\section{Introduction}
Let $\Omega $ be a region in $\mathbb{R}^{n}$. It is said that a finite or countably infinite number of squares in $\mathbb{R}^{n}$ are {\it packed} into $\Omega$ if they are included in $\Omega$ with disjoint interiors. If the region $\Omega$ is covered by  these squares up to null subsets, then we say that $\Omega$ is {\it packed perfectly} by these squares. We also call this situation a {\it perfectly packing} of $\Omega$ by squares.  

Meir and Moser \cite{MM} presented the problem of whether rectangles of dimension $\frac{1}{n}\times \frac{1}{n+1}$ for $n\geq 1$ can be packed perfectly into a square of area $1$. They also posed the question whether squares of sidelength $n^{-1}$ for $n\geq 2$ can be packed perfectly into a square of area $\frac{\pi ^{2}}{6} -1$. Currently these questions are still unsolved. One of the best current results on these problems is the work of Paulhus \cite{P}, who obtained an algorithm for packing  squares of sidelength  $\frac{1}{n}\times \frac{1}{n+1}$ for $n\geq 1$ with disjoint interiors into a square of area $1+\frac{1}{10^{9}+1}$, and that for packing squares of sidelength $n^{-1}$ for $n\geq 2$ with disjoint interiors into a square of area $\frac{\pi ^{2}}{6}-1 +\frac{1}{1244918662}$.  There were some incorrect lemmas in Paulhus' paper. However, Grzegorek and Januszewski \cite{GJ} modified Paulhus' argument and gave completely correct proofs.

Another approach for the second question of Meir and Moser is to consider whether squares of sidelength $n^{-t}$ for $n\geq 2$ can be packed with disjoint interiors into a square of area $\zeta (2t)-1$ for $t>1/2$. The goal is to prove that there exists such a packing for $t=1$. Januszewski and Zielonka \cite{JZ} proved that there is a proper packing for $1/2<t\leq 2/3$ by combining ideas in \cite{C}, \cite{J}, \cite{W} and Paulhus' algorithm \cite{P} for packing squares. They also considered the 3-dimensional case, i.e.,  the problem of perfectly packing  a cube by cubes. Furthermore, there are several papers considering perfect packing of the $d$-dimensional cubes of sidelength $1^{-t}$, $2^{-t}$, $3^{-t}, \ldots$ for special values of $t$ (see \cite{JZ2} and \cite{J2}, for example).

Recently Tao \cite{T} conditionally improved their work. He proved that for any $1/2<t<1$, there exists a large positive integer $N_{0}$ depending on $t$ such that for any integer $n_{0}\geq N_{0}$,  squares of sidelength $n^{-t}$ for $n\geq n_{0}$ can be packed perfectly into a square of area $\sum _{n=n_{0}}^{\infty}n^{-2t}$. His new idea is to arrange the squares in near-lattice formations, which enables us to take $t$ arbitrarily close to $1$. He also remarked that his method will also be applicable to the problem of packing squares of dimension $\frac{1}{n^{t}}\times \frac{1}{(n+1)^{t}}$ into a square of area $\sum _{n=n_{0}}^{\infty}\frac{1}{n^{t}(n+1)^{t}}$.  More recently, McClenagan \cite{Mc} obtained the 3-dimensional analogue of Tao's result. His theorem is also an improvement of the result of Januszewski and Zielonka \cite{JZ} mentioned above.

In this paper, for $q> r \geq 0$, we call the sequence $\{(qn+r)^{-1} \}$ an {\it AP-harmonic} sequence, and call the sequence of reciprocals of prime numbers the {\it P-harmonic} sequence. Furthermore, we call the sequence of reciprocals of twin primes the {\it TP-harmonic} sequence. 

We use a non-standard notation to estimate the size of functions effectively. When two functions $F$, $G$ satisfy $|F(x)|\leq G(x)$ over some specified range, we denote this relation by $F(x)=O_{1}(G(x))$.

The purpose of this paper is to generalize Tao's theorem to the problem of packing of squares of sidelength $f(n)^{-t}$ for $1/2<t<1$ and give some effective lower bounds for the size of $N_{0}$ depending on $t$ in some specific cases. The result includes (an effective version of) Tao's theorem. We will prove the  main theorem below in Section 2. In Section 3,   as a corollary of Theorem 1.1,  we will prove a result on perfectly packing  a square by squares of nearly AP-harmonic sidelength. In Section 4, we will prove a similar result in the case of packing of squares of nearly P-harmonic sidelength. Finally, in Section 5, we will give an ineffective version of Theorem 1.1 and  give a packing of squares of nearly TP-harmonic sidelength into a slightly larger square than the expected size.

\begin{thm}
Let $1/2<t<1$ and put $\delta :=1-t$. Suppose that there exist positive constants $c_{i}, d_{i} (i=1,2), c_{\lambda ,\nu}, l, K$,  functions $\nu _{i}, \lambda _{i}: \mathbb{N}\to \mathbb{R}_{>0}$ $(i=1,2)$, positive functions $\xi _{i}$, $\eta _{i}$ on the open interval $(1/2, 1)$ $(i=1,2)$, positive integers $M$, $N_{0}$, $N_{1}$ $(N_{1}<N_{0})$ and a smooth increasing function $f: \mathbb{R}_{\geq 1} \to \mathbb{R}_{>0}$ such that for any positive integers $n_{0}\geq N_{0}$, $n_{1}\geq N_{1}$, the following ten conditions hold.

\begin{equation}
\label{16}
f(n_{0})\geq K f((1+l)n_{0}), \quad K\geq (2/3)^{\frac{1}{2}},
\end{equation}

\begin{equation}
\label{20}
c_{1}\frac{\xi _{1}(t)\nu _{1}(n_{1})}{f(n_{1})^{t+\delta t}} \leq \sum _{n=1}^{n_{1}-1}\frac{1}{f(n)^{t+\delta t}} \leq c_{2}\frac{\xi _{2}(t)\nu _{2}(n_{1})}{f(n_{1})^{t+\delta t}},
\end{equation}

\begin{equation}
\label{21}
d_{1}\frac{\eta _{1}(t)\lambda _{1}(n_{1})}{f(n_{1})^{2t}} \leq \sum _{n=n_{1}}^{\infty}\frac{1}{f(n)^{2t}} \leq d_{2}\frac{\eta _{2}(t)\lambda _{2}(n_{1})}{f(n_{1})^{2t}} ,
\end{equation}

\begin{equation}
\label{25}
\lambda _{1}(n_{0})\geq c_{\lambda ,\nu}\nu _{2}(n_{0}),
\end{equation}

\begin{equation}
\label{34}
\lambda _{2}(n_{0})\leq d_{2}^{-1}\eta _{2}(t)^{-1}f(n_{0})^{2t},
\end{equation}

\begin{equation}
\label{M}
M\geq \max \left\{ 4^{\frac{1}{1-t}} \left( \frac{c_{2}\xi _{2}(t)}{c_{\lambda ,\nu}d_{1}\eta _{1}(t)} \right)^{\frac{2}{t(1-t)}}, \left( \frac{50}{K^{t(2-t)}} \right)^{\frac{2}{1-t}}   \right\},
\end{equation}

\begin{equation}
\label{15}
\underset{x\in [n_{0}, n_{0}+9M^{2}]}{\max}f^{'}(x)\leq \frac{1}{4752}M^{-4}f(n_{0}),
\end{equation}

\begin{equation}
\label{15.5}
f(n_{0}+9M^{2})\leq (264M^{2})^{\frac{1}{t}}f(n_{0}),
\end{equation}

\begin{equation}
\label{31}
9Mf(n_{0})^{t}\leq ln_{0},
\end{equation}

\begin{equation}
\label{33}
\nu _{1}(n_{0})\lambda _{2}(n_{0})^{-\frac{1+\delta}{2}}\geq \frac{(d_{2}\eta _{2}(t))^{\frac{1+\delta}{2}}}{c_{1}\xi _{1}(t)}M^{1-\frac{\delta}{2}}.
\end{equation}

Then, for any $n_{0}\geq N_{0}$,  squares of sidelength $f(n)^{-t}$ for $n \geq n_{0}$ can be packed perfectly into a square of area $\sum _{n=n_{0}}^{\infty}f(n)^{-2t}$. 
\end{thm}

\section{Proof of the main theorem}
Given a rectangle $R$, we define the {\it width}  $w(R)$ of $R$ to be the smaller of two sidelengths of $R$, and define the ${\it height}$ $h(R)$ of $R$ to be the larger one. If $R$ is a square, then $w(R)=h(R)$. The area  of $R$ is equal to $w(R)h(R)$. Given a family of rectangles ${\cal R}$, we define the {\it total area} of ${\cal R}$ by
\[
\mathrm{area}({\cal R})=\sum _{R\in {\cal R}}w(R)h(R),
\]
and define its {\it unweighted total perimeter}  by 
\[
\mathrm{perim}({\cal R})=2\sum _{R\in {\cal R}}(w(R)+h(R)).
\]

\subsection{Efficiently packing a small rectangle of bounded eccentricity}
\begin{prop}
Let $1/2<t<1$ and $M$ be a positive integer. Suppose that a smooth increasing function   $f: \mathbb{R}_{\geq 1}\to \mathbb{R}_{>0}$ and any positive integer $n_{0}\geq N_{0}$ satisfy (\ref{15}) and (\ref{15.5}). Let $R$ be a rectangle with 
\begin{equation}
\label{12}
Mf(n_{0})^{-t}\leq w(R) \leq h(R) \leq 3Mf(n_{0})^{-t}.
\end{equation}
Then, there exists a positive integer $n_{0}^{'}$ satisfying $M^{2}\leq n_{0}^{'}-n_{0}\leq 9M^{2}$ such that $R$ is perfectly packed by squares of sidelength $f(n)^{-t}$ $(n_{0}\leq n <n_{0}^{'})$ and a family of rectangles ${\cal R}$ with disjoint interiors whose total perimeter satisfies
\begin{equation}
\label{15.75}
\mathrm{perim}({\cal R})\leq 25Mf(n_{0})^{-t}
\end{equation}
and the width of each rectangle is at most $2f(n_{0})^{-t}$. 
\end{prop}
\noindent
{\it Proof of  Proposition 2.1}.\\
We may assume that $R$ is the rectangle $[0, w(R)]\times [0, h(R)]$. Then by (\ref{12}), there exist positive integers $M_{1}, M_{2}$ $(M\leq M_{1}\leq M_{2} \leq 3M)$ satisfying 
\[
M_{1}f(n_{0})^{-t}\leq w(R)<(M_{1}+1)f(n_{0})^{-t}, \;
M_{2}f(n_{0})^{-t}\leq h(R)<(M_{2}+1)f(n_{0})^{-t}.
\]
Put $n_{0}^{'}:=n_{0}+M_{1}M_{2}$. Then $M^{2}\leq n_{0}^{'}-n_{0}\leq 9M^{2}$. Rewrite the set $\{ n \; |\; n_{0}\leq n <n_{0}+M_{1}M_{2} \}$ as $\{ n_{i,j}\; |\;  0\leq i <M_{1}, 0\leq j <M_{2} \}$, where 
\[
n_{i,j}:=n_{0}+jM_{1}+i.
\]
Let $S_{i,j}$ be a square of sidelength $f(n_{i,j})^{-t}$. Following the idea in \cite{T}, we arrange this square by 
\begin{equation}
\label{4}
S_{i,j}=[x_{i,j},x_{i,j}+f(n_{i,j})^{-t}]\times [y_{i,j},y_{i,j}+f(n_{i,j})^{-t}],
\end{equation}
where 
\begin{equation}
\label{5}
x_{i,j}:=w(R)-\sum _{i^{'}=i}^{M_{1}}f(n_{i^{'},j})^{-t},
\end{equation}
\begin{equation}
\label{6}
y_{i,j}:=\sum _{j^{'}=0}^{j-1}f(n_{i,j^{'}})^{-t} \; (j\geq 1), \quad y_{i,0}:=0
\end{equation}
for $1 \leq i <M_{1}, 1\leq j <M_{2}$. We need some analysis on $f$. By mean value theorem, 
\[
f(n_{i,j})=f(n_{0}+jM_{1}+i)=f(n_{0})+f^{'}(c)(jM_{1}+i)
\]
for some $c\in (n_{0},n_{0}+ jM_{1}+i)$. Notice that $n_{0}+ jM_{1}+i \leq n_{0}+9M^{2}$. We denote the maximum value of $f^{'}(x)$ in $[n_{0},n_{0}+9M^{2}]$ by ${\cal D}_{n_{0},M}$. Then
\begin{align*}
f(n_{i,j})=&f(n_{0})+O_{1}(9M^{2}{\cal D}_{n_{0},M}) \\
=& f(n_{0})\left(1+O_{1}\left(\frac{9M^{2}{\cal D}_{n_{0},M}}{f(n_{0})}  \right)  \right).
\end{align*}
By (\ref{15}), we have
\begin{equation}
\label{7}
\frac{9M^{2}{\cal D}_{n_{0},M}}{f(n_{0})}<\frac{1}{100}.
\end{equation}
Hence
\[
f(n_{i,j})^{-t}= f(n_{0})^{-t}\left(1+O_{1}\left(\frac{9M^{2}{\cal D}_{n_{0},M}}{f(n_{0})}  \right)  \right)
\]
uniformly for $1/2<t<1$. Then by (\ref{5}) and (\ref{6}), it follows that 
\begin{equation}
\label{8}
x_{i,j}=w(R)-f(n_{0})^{-t}\left(M_{1}-i+1+ O_{1}\left(\frac{27M^{3}{\cal D}_{n_{0},M}}{f(n_{0})}  \right)    \right),
\end{equation}
\begin{equation}
\label{9}
y_{i,j}=f(n_{0})^{-t}\left(j+O_{1}\left(\frac{27M^{3}{\cal D}_{n_{0},M}}{f(n_{0})}  \right)     \right)
\end{equation}
for $0\leq i <M_{1}, 0\leq j <M_{2}$. 

\begin{lem}
$(1)$ $S_{i,j} (0\leq i <M_{1}, 0\leq j <M_{2})$ are included in $R$. \\
$(2)$ If $(i,j)\neq (i^{'},j^{'})$, the interiors of $S_{i,j}$ and $S_{i^{'},j^{'}}$ are disjoint.
\end{lem}
\begin{proof}
(1) Since $f$ is increasing, we see that
\[
0\leq w(R)-M_{1}f(n_{0})^{-t}\leq x_{i,j} \leq x_{i,j}+f(n_{i,j})^{-t}\leq w(R),
\]
\[
0\leq y_{i,j}\leq y_{i,j}+f(n_{i,j})^{-t}\leq M_{2}f(n_{0})^{-t}\leq h(R).
\]
Therefore, $S_{i,j}$ is included in $R$. \\
(2) If $i\geq i^{'}$, $j<j^{'}$, then $y_{i,j}+f(n_{i,j})^{-t}\leq y_{i,j^{'}}\leq y_{i^{'},j^{'}}$. Hence the interior of $S_{i,j}$ lies bottom to the interior of $S_{i^{'},j^{'}}$. Similarly if $i^{'}\geq i$, $j^{'}<j$. \\
If $i<i^{'}$, $j\leq j^{'}$, then $x_{i,j} +f(n_{i,j})^{-t}\leq x_{i^{'},j}\leq x_{i^{'},j^{'}}$. Hence the interior of $S_{i,j}$ lies left to the interior of $S_{i^{'},j^{'}}$. Similarly if $i^{'}<i$, $j^{'}\leq j$.
\end{proof}
\begin{lem}
Suppose 
\begin{equation}
\label{X}
9M^{2}\underset{0\leq x \leq 9M^{2}}{\max}f^{'}(n_{0}+x)f(n_{0}+x)^{-t-1}\leq f(n_{0}+9M^{2})^{-t}.
\end{equation}
Then, for  $0\leq i<M_{1}-1$ and $0\leq j <M_{2}-1$,  the squares $S_{i,j}$, $S_{i+1,j}$, $S_{i,j+1}$ and $S_{i+1,j+1}$ surround a rectangle
\begin{equation}
\label{10}
[x_{i+1, j}, x_{i+1,j+1}]\times [y_{i+1,j+1},y_{i,j+1}].
\end{equation}
\end{lem}


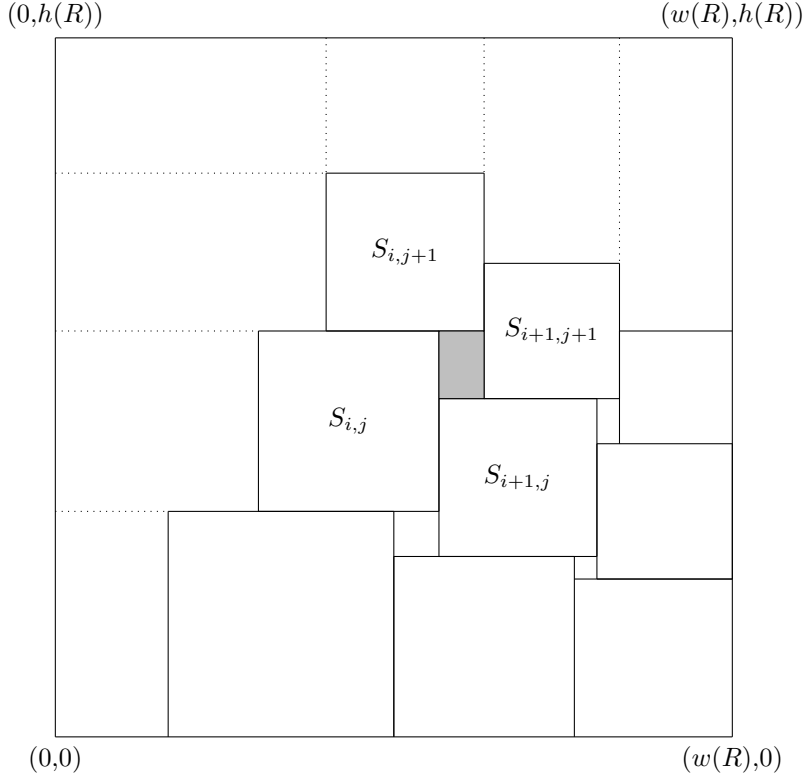
\begin{tikzpicture}[scale=3]
\draw(0,0)--(3,0);
\draw(0,0)--(0,3.1);
\draw( 0 ,3.1 )--(3.0  , 3.1 );
\draw( 3.0 ,0 )--(3.0  ,3.1  );
\draw( 0.5 , 0)--(  0.5,  1);
\draw( 0.5 ,1 )--( 1.5 , 1 );
\draw( 1.5 ,1 )--( 1.5 ,  0);
\draw( 1.5 , 0)--( 1.5 ,  0.8);
\draw(  1.5, 0.8)--( 2.3 , 0.8 );
\draw( 2.3 ,0.8 )--( 2.3 ,0  );
\draw( 2.3 ,0.7 )--( 3.0 , 0.7 );
\draw( 0.9 ,1 )--( 0.9 , 1.8 );
\draw(  0.9,1.8 )--( 1.7 , 1.8 );
\draw( 1.7 ,1.8 )--( 1.7 , 1 );
\draw( 1.7 ,1 )--( 0.9 ,  1);
\draw( 1.7 ,0.8 )--( 1.7 , 1.5 );
\draw( 1.7 ,1.5 )--( 2.4 ,1.5  );
\draw(  2.4, 1.5)--( 2.4 ,0.8  );
\draw( 2.4, 0.8 )--( 1.7 , 0.8 );
\draw( 2.4 , 0.7)--( 2.4 ,1.3  );
\draw(2.4  , 1.3)--(3.0  ,1.3  );
\draw(3.0  , 1.3)--(3.0  ,0.7  );
\draw( 3.0 , 0.7)--(2.4  ,0.7  );
\draw(1.2  ,1.8 )--( 1.2 ,  2.5);
\draw(  1.2,2.5 )--(1.9  , 2.5 );
\draw( 1.9 , 2.5)--(1.9  , 1.8 );
\draw( 1.9 ,1.8 )--(1.2  ,1.8  );
\draw( 1.9 ,1.5 )--( 1.9 ,2.1  );
\draw(1.9  , 2.1)--(2.5  , 2.1 );
\draw( 2.5 , 2.1)--(2.5  ,1.5  );
\draw( 2.5 , 1.5)--(1.9  , 1.5 );
\draw( 2.5 , 1.3)--( 2.5 ,1.8  );
\draw(2.5  , 1.8)--(3.0  ,1.8  );
\draw[dotted](0,1)--(0.5,1);
\draw[dotted]( 0 ,1.8   )--( 0.9  ,  1.8 );
\draw[dotted](   0,  2.5 )--( 1.2  ,2.5   );
\draw[dotted](  1.2 , 3.1  )--( 1.2  , 2.5  );
\draw[dotted](  1.9 ,  3.1 )--( 1.9  ,  2.5 );
\draw[dotted](  2.5 , 3.1  )--( 2.5  , 2.1  );
\draw(0,0)node[below]{(0,0)};
\draw(3,0)node[below]{($w(R)$,0)};
\draw(0,3.1)node[above]{(0,$h(R)$)};
\draw(3,3.1)node[above]{($w(R)$,$h(R)$)};
\draw(1.3,1.4)node{$S_{i,j}$}; 
\draw(2.05,1.15)node{$S_{i+1,j}$}; 
\draw(2.2,1.8)node{$S_{i+1,j+1}$}; 
\draw(1.55,2.15)node{$S_{i,j+1}$}; 
\filldraw[fill=lightgray](1.7,1.5)--(1.9,1.5)--(1.9,1.8)--(1.7,1.8)--(1.7,1.5);
\end{tikzpicture}


\begin{proof}
This fact follows from the following equalities and inequalities. 
\begin{align*}
&x_{i+1,j}=x_{i,j}+f(n_{i,j})^{-t}, \\
&\underline{y_{i,j}<y_{i+1,j}+f(n_{i+1,j})^{-t}}<y_{i,j}+f(n_{i,j})^{-t}, \\
&x_{i,j}<\underline{x_{i,j+1}<x_{i,j}+f(n_{i,j})^{-t}} \\
&y_{i,j+1}=y_{i,j}+f(n_{i,j})^{-t} \\
&x_{i+1,j+1}=x_{i,j+1}+f(n_{i,j+1})^{-t}, \\
&y_{i+1,j+1}<\underline{y_{i,j+1}<y_{i+1,j+1}+f(n_{i+1,j+1})^{-t}}, \\
&x_{i+1,j}<\underline{x_{i+1,j+1}<x_{i+1,j}+f(n_{i+1,j})^{-t}}, \\
&y_{i+1,j+1}=y_{i+1,j}+f(n_{i+1,j})^{-t}.
\end{align*}
Among these relations, the underlined four (essentially two) inequalities are nontrivial, and others easily follow from the definitions of $x_{i,j}$ and $y_{i,j}$.  By definition, 
\[
y_{i,j}-y_{i+1,j}=\sum _{j^{'}=0}^{j-1}\left( f(n_{0}+j^{'}M_{1}+i)^{-t}-f(n_{0}+j^{'}M_{1}+i+1)^{-t} \right).
\]
By mean value theorem, we have
\begin{align*}
&  f(n_{0}+j^{'}M_{1}+i)^{-t}-f(n_{0}+j^{'}M_{1}+i+1)^{-t} \\
& \quad =-\frac{\partial}{\partial u}f(n_{0}+j^{'}M_{1}+u) ^{-t}|_{u=c_{i,j^{'}}} \\
& \quad =tf^{'}(n_{0}+j^{'}M_{1}+c_{i,j^{'}})f(n_{0}+j^{'}M_{1}+c_{i,j^{'}})^{-t-1}
\end{align*}
for some $c_{i,j^{'}}\in (i, i+1)$. Hence
\begin{align*}
y_{i,j}-y_{i+1,j}=& t\sum _{j^{'}=0}^{j-1}f^{'}(n_{0}+j^{'}M_{1}+c_{i,j^{'}})f(n_{0}+j^{'}M_{1}+c_{i,j^{'}})^{-t-1} \\
& \leq 3M \underset{0\leq x \leq 9M^{2}}{\max}f^{'}(n_{0}+x)f(n_{0}+x)^{-t-1}.
\end{align*}
On the other hand, since $f$ is increasing, $f(n_{i+1,j})^{-t}$ is equal or larger than \\ $f(n_{0}+9M^{2})^{-t}$. Hence if 
\begin{equation}
\label{Y}
 3M\underset{0\leq x \leq 9M^{2}}{{\max}}f^{'}(n_{0}+x)f(n_{0}+x)^{-t-1}\leq f(n_{0}+9M^{2})^{-t}
\end{equation}
holds, then the first and third underlined inequalities hold. Moreover, the condition (\ref{Y}) is weaker than (\ref{X}). Hence the first and third underlined inequalities follow from (\ref{X}). We now prove the second and fourth inequalities assuming (\ref{X}). Again by definition,
\[
x_{i,j+1}-x_{i,j}=\sum _{i^{'}=i}^{M_{1}}\left(f(n_{0}+jM_{1}+i^{'})^{-t}-f(n_{0}+(j+1)M_{1}+i^{'})^{-t} \right).
\]
By mean value theorem,
\begin{align*}
& f(n_{0}+jM_{1}+i^{'})^{-t}-f(n_{0}+(j+1)M_{1}+i^{'})^{-t} \\
& \quad =-\frac{\partial}{\partial u}f(n_{0}+uM_{1}+i^{'})^{-t}|_{u=c_{i^{'},j}} \\
& \quad =tM_{1}f^{'}(n_{0}+c_{i^{'},j}M_{1}+i^{'})f(n_{0}+c_{i^{'},j}M_{1}+i^{'})^{-t-1}.
\end{align*}
for some $c_{i^{'},j}\in (j, j+1)$. Hence by (\ref{X}), 
\begin{align*}
x_{i,j+1}-x_{i,j}=& tM_{1} \sum _{i=i^{'}}^{M_{1}}f^{'}(n_{0}+c_{i^{'},j}M_{1}+i^{'})f(n_{0}+c_{i^{'},j}M_{1}+i^{'})^{-t-1} \\
\leq & 9M^{2} \underset{0\leq x \leq 9M^{2}}{{\max}}f^{'}(n_{0}+x)f(n_{0}+x)^{-t-1}\leq f(n_{0}+9M^{2})^{-t}.
\end{align*}
On the other hand, $f(n_{i,j})^{-t}$ is equal or larger than $f(n_{0}+9M^{2})^{-t}$. Thus we obtain the second and fourth underlined inequalities under the assumption of (\ref{X}). 
\end{proof}

By (\ref{15}), the left hand side of (\ref{X}) is at most $(1/264M^{2})f(n_{0})^{-t}$. Hence (\ref{X}) is satisfied if (\ref{15.5}) holds.  Summing up, the rectangle $R$ is perfectly packed by the following five types of rectangles: \\
Type I. The squares $S_{i,j}$ $(0\leq i <M_{1}, 0\leq j <M_{2})$. \\
Type II. The rectangles  (\ref{10})  surrounded by Type I squares. By (\ref{8}) and (\ref{9}), the perimeters of these surrounded rectangles are $O_{1}\left( \frac{216M^{3}{\cal D}_{n_{0},M}}{f(n_{0})^{1+t}} \right)$ and the number of these rectangles is at most $9M^{2}$. \\
Type III. The rectangles $[0, x_{0,j}]\times [y_{0,j},y_{0,j}+f(n_{0,j})^{-t}]$ $(0\leq j <M_{2})$. By (\ref{8}), the perimeters of these rectangles are at most $4f(n_{0})^{-t}+\frac{54M^{3}{\cal D}_{n_{0},M}}{f(n_{0})^{t+1}}$ and the number of these rectangles is at most $3M$. \\
Type IV. The rectangles $[x_{i, M_{2}-1}, x_{i, M_{2}-1}+f(n_{i,M_{2}-1})^{-t}]\times [y_{i, M_{2}-1}+f(n_{i, M_{2}-1})^{-t}, h(R)]$ $(0\leq i<M_{1})$. The perimeters of these rectangles are at most $4f(n_{0})^{-t}+\frac{54M^{3}{\cal D}_{n_{0},M}}{f(n_{0})^{t+1}}$, and the number of these rectangles is at most $3M$. \\
Type V. The rectangle $[0, x_{0, M_{2}-1}] \times [y_{0,M_{2}-1}+f(n_{0, M_{2}-1})^{-t}, h(R)]$. The perimeter of this rectangle is at most $ 4f(n_{0})^{-t}+\frac{108M^{3}{\cal D}_{n_{0},M}}{f(n_{0})^{1+t}} $. \\
Put ${\cal R}:=R\backslash \{S_{i,j} \}_{i,j}$. By the above argument, it follows that the widths of the rectangles contained in ${\cal R}$  exceed neither $54M^{3}{\cal D}_{n_{0},M}f(n_{0})^{-t-1}$ nor $f(n_{0})^{-t}+  27M^{3}{\cal D}_{n_{0},M}f(n_{0})^{-t-1}$, so the widths are at most $2f(n_{0})^{-t}$ by (\ref{15}).     Moreover, we see that the total perimeter of ${\cal R}$ is
\begin{equation}
\label{11}
\begin{aligned}
\mathrm{perim}({\cal R})\leq &  \frac{216M^{3}{\cal D}_{n_{0}, M}}{f(n_{0})^{1+t}}\times 9M^{2}+\left(4f(n_{0})^{-t}+ \frac{54M^{3}{\cal D}_{n_{0}, M}}{f(n_{0})^{1+t}}   \right)\times 6M\\
& \quad +\left(4f(n_{0})^{-t}+ \frac{108M^{3}{\cal D}_{n_{0}, M}}{f(n_{0})^{1+t}}   \right) \\
\leq &\frac{2376M^{5}{\cal D}_{n_{0},M}}{f(n_{0})^{t+1}}+\frac{49}{2}Mf(n_{0})^{-t} \\
\leq &25 Mf(n_{0})^{-t}.
\end{aligned}
\end{equation}
Thus the proof of Proposition 2.1 is completed.  \hspace{6.8cm} $\Box$

\subsection{Perfectly packing some families of rectangles}
Put $\delta =1-t$ and define the {\it weighted total perimeter} of ${\cal R}$ by
\[
\mathrm{perim}_{\delta}({\cal R})=\sum _{R\in {\cal R}}w(R)^{\delta}h(R).
\]
\begin{prop}
Suppose that $M$, $f$, $K$, $l$ and $1/2<t<1$ satisfy conditions (\ref{16}), (\ref{20}),  (\ref{21}), (\ref{25}) and (\ref{M})    for any $n_{0}\geq N_{0}$, $n_{1}\geq N_{1}$. Fix positive integers $n_{\max} \geq n_{0} (\geq N_{0})$. Let ${\cal R}$ be a family of finite number of rectangles with disjoint interiors satisfying 
\begin{equation}
\label{17}
\mathrm{area} ({\cal R})=\sum _{n=n_{0}}^{\infty}\frac{1}{f(n)^{2t}}, 
\end{equation}
\begin{equation}
\label{18}
\mathrm{perim}_{\delta}({\cal R})\leq M^{-1+\frac{\delta}{2}}\sum _{n=1}^{n_{0}-1}\frac{1}{f(n)^{t+\delta t}}
\end{equation}
and 
\begin{equation}
\label{19}
\sup _{R\in {\cal R}}h(R) \leq 1.
\end{equation}
Then, squares of sidelength $f(n)^{-t}$ $(n_{0}\leq n <n_{\max})$ can be packed into $\bigcup _{R\in {\cal R}}R$. 
\end{prop}
\begin{proof}
Fix $n_{\max}$ and prove the statement by a downward induction on $n_{0}$. The above proposition is trivial if $n_{\max}=n_{0}$. So we suppose that $n_{\max}>n_{0}$ and that the claim has already been proved for larger $n_{0}$. 
 Then, by (\ref{18}) and (\ref{20}), we have
 \begin{equation}
 \label{18.5}
 \sum _{R\in {\cal R}}w(R)^{\delta}h(R)\leq M^{-1+\frac{\delta}{2}}\frac{c_{2}\xi _{2}(t)\nu _{2}(n_{0})}{f(n_{0})^{t+\delta t}}.
 \end{equation}
 On the other hand, by (\ref{17}) and (\ref{21}), we have
 \begin{equation}
 \label{23}
 \sum _{R\in {\cal R}}w(R)h(R) \geq \frac{d_{1}\eta _{1}(t)\lambda _{1}(n_{0})}{f(n_{0})^{2t}}.
 \end{equation}
Hence by the pigeonhole principle, there exists a rectangle $R \in {\cal R}$ that satisfies 
\[
w(R)^{1-\delta}\geq \frac{d_{1}\eta _{1}(t)\lambda _{1}(n_{0})f(n_{0})^{-2t}}{M^{-1+\frac{\delta}{2}}c_{2}\xi _{2}(t)\nu _{2}(n_{0})f(n_{0})^{-t-\delta t}}.
\]
Therefore, we have
\begin{align*}
w(R)&\geq (M^{1-\frac{\delta}{2}})^{\frac{1}{1-\delta}}\left(\frac{d_{1}\eta _{1}(t)\lambda _{1}(n_{0})}{c_{2}\xi _{2}(t)\nu _{2}(n_{0})}   \right)^{\frac{1}{1-\delta}}f(n_{0})^{-t} \\
& \geq M^{1+\frac{\delta}{2}} \left(\frac{d_{1}\eta _{1}(t)\lambda _{1}(n_{0})}{c_{2}\xi _{2}(t)\nu _{2}(n_{0})}   \right)^{\frac{1}{t}}f(n_{0})^{-t},
\end{align*}
since $(1-\frac{\delta}{2})/(1-\delta)>1+\frac{\delta}{2}$. Since $\lambda _{1}$ and $\nu _{2}$ satisfy (\ref{25}), we have
\[
w(R)\geq M^{1+\frac{\delta}{2}} \left(\frac{c_{\lambda ,\nu}d_{1}\eta _{1}(t)}{c_{2}\xi _{2}(t)}   \right)^{\frac{1}{t}}f(n_{0})^{-t}.
\]
Since $M$ satisfies (\ref{M}), it follows that 
\[
M^{\frac{\delta}{2}}\left( \frac{c_{\lambda ,\nu}d_{1}\eta _{1}(t)}{c_{2}\xi _{2}(t)} \right)^{\frac{1}{t}}\geq 2.
\]
Therefore,
\begin{equation}
\label{27}
h(R)\geq w(R) \geq 2Mf(n_{0})^{-t}.
\end{equation}
We cut the rectangle $R$ into a rectangle $R_{0}$ of dimension $(w(R)-Mf(n_{0})^{-t})\times h(R)$ and a rectangle $R_{*}$ of dimension $Mf(n_{0})^{-t}\times h(R)$. Then cut off the squares of sidelength $Mf(n_{0})^{-t}$ from $R_{*}$ until the height of the remaining rectangle is below $2Mf(n_{0})^{-t}$. Thus $R_{*}$ is decomposed into rectangles $R_{1}, \ldots ,R_{m}$. These rectangles are of dimension $Mf(n_{0})^{-t}\times h(R_{i})$, where 
\begin{equation}
\label{insufficient}
Mf(n_{0})^{-t}\leq h(R_{i})<2Mf(n_{0})^{-t} \quad (i=1, \ldots ,m)
\end{equation}
with disjoint interiors, and 
\[
\sum _{i=1}^{m}h(R_{i})=h(R).
\]
Since $h(R)\leq 1$, it follows that 
\begin{equation}
\label{28}
m\leq M^{-1}f(n_{0})^{t}.
\end{equation}
This gives a perfectly packing of $R $
\[
R=R_{1}\cup R_{2}\cup \cdots \cup R_{m}.
\]
We adapt Proposition 2.1 $m$ times. Then there exists a sequence of positive integers $n_{0}=n_{0}^{'}<n_{1}^{'}<\cdots <n_{m}^{'}$ such that 
\begin{equation}
\label{29}
M^{2}\leq n_{i+1}^{'}-n_{i}^{'}\leq 9M^{2} \quad (0\leq i \leq m-1)
\end{equation}
and each $R_{i}$ is perfectly packed by squares of sidelength $f(n)^{-t}$ $(n_{i-1}^{'}\leq n <n_{i}^{'})$ and a family  ${\cal R}_{i}$ of finite number of rectangles with disjoint interiors satisfying
\begin{equation}
\label{30}
\mathrm{perim}({\cal R}_{i})\leq 25Mf(n_{0})^{-t}
\end{equation}
and the width of each rectangle in ${\cal R}_{i}$ is at most $2f(n_{0})^{-t}$, provided that the height of $R_{i}$ satisfies
\begin{equation}
\label{sufficient}
Mf(n_{i-1}^{'})^{-t}\leq h(R_{i}) \leq 3Mf(n_{i-1}^{'})^{-t}
\end{equation}
for $i=1, \ldots ,m$. Let us confirm that (\ref{sufficient}) is satisfied. Since $Mf(n_{i-1}^{'})^{-t}\leq Mf(n_{0})^{-t}$, $Mf(n_{i-1}^{'})^{-t}\leq h(R_{i})$ is satisfied. By (\ref{28}) and (\ref{29}),
\[
n_{i}^{'}\leq n_{0}+9Mf(n_{0})^{t}.
\]
Combining this with (\ref{31}), we have
\begin{equation}
\label{32}
n_{0}\leq n_{i}^{'}\leq (1+l)n_{0}
\end{equation}
for $i=1, \ldots ,m$. Since $h(R_{i})$ satisfies (\ref{insufficient}), to prove $h(R_{i})\leq  3Mf(n_{i-1}^{'})^{-t}$, it suffices to see $2Mf(n_{0})^{-t}\leq 3Mf(n_{i-1}^{'})^{-t}$, which is equivalent to $f(n_{0})\geq (2/3)^{\frac{1}{t}}f(n_{i-1}^{'})$. For $i=1, \ldots ,m$, these inequalities are  valid because of (\ref{16}) and (\ref{32}). Hence (\ref{sufficient}) holds and the inductive argument can be applied to the rectangles $R_{1}, \ldots ,R_{m}$. Put
\[
{\cal R}^{'}:=({\cal R}\backslash \{R \}) \cup \{R_{0} \} \cup \bigcup _{i=1}^{m}{\cal R}_{i}.
\]
By the above argument, $\bigcup _{R^{'}\in {\cal R}}R^{'}$ is perfectly packed by squares of sidelength $f(n)^{-t}$, $n_{0}\leq n<n_{m}^{'}$ and rectangles in ${\cal R}^{'}$. If $n_{m}^{'}\geq n_{\max}$, the induction process is completed. So we assume $n_{m}^{'}<n_{\max}$. Before the evaluation of $\textrm{perim}_{\delta}({\cal R}^{'})$, we see that for $i=1, \ldots ,m$,
\begin{equation}
\label{fn0}
f(n_{0})^{-t-\delta t}\leq M^{-2}K^{-t-\delta t}\sum _{n=n_{i-1}^{'}}^{n_{i}^{'}-1}f(n)^{-t-\delta t}
\end{equation}
holds, because $f(n_{0})\geq Kf(n)$ for $n_{i-1}^{'}\leq n <n_{i}^{'}-1$. By (\ref{30}), (\ref{18}), (\ref{fn0}) and $w(R^{'})\leq 2f(n_{0})^{-t}$ for $R^{'}\in {\cal R}_{i}$, it follows that
\begin{align*}
\mathrm{perim}_{\delta}({\cal R}^{'}) & =\mathrm{perim}_{\delta}({\cal R}) -w(R)^{\delta}h(R)+w(R_{0})^{\delta}h(R)+\sum _{i=1}^{m}\sum _{R^{'}\in {\cal R}_{i}}w(R^{'})^{\delta }h(R^{'})  \\
& \leq \mathrm{perim}_{\delta}({\cal R}) +\sum _{i=1}^{m}O_{1}(2^{\delta}f(n_{0})^{-\delta t}\mathrm{perim}({\cal R}_{i})) \\
&  \leq  \mathrm{perim}_{\delta}({\cal R}) +2^{\delta}\cdot 25M\sum _{i=1}^{m}O_{1}(f(n_{0})^{-t-\delta t}) \\
& \leq \mathrm{perim}_{\delta}({\cal R}) +\frac{50}{K^{t+\delta t}M}\sum _{i=1}^{m}O_{1}\left( \sum _{n=n_{i-1}^{'}}^{n_{i}^{'}-1}f(n)^{-t-\delta t} \right) \\
& \leq M^{-1+\frac{\delta}{2}}\sum _{n=1}^{n_{0}-1}\frac{1}{f(n)^{t+\delta t}} +\frac{50}{K^{t+\delta t}M}\sum _{n=n_{0}}^{n_{m}^{'}-1}\frac{1}{f(n)^{t+\delta t}}.
\end{align*}
Since $M\geq (50/K^{t(2-t)})^{\frac{2}{1-t}}$ by (\ref{M}), we have
\[
\mathrm{perim}_{\delta}({\cal R}^{'})\leq M^{-1+\frac{\delta}{2}}\sum _{n=1}^{n_{m}^{'}-1}\frac{1}{f(n)^{t+\delta t}}.
\]
Therefore, ${\cal R}^{'}$ satisfies (\ref{18}) with $n_{0}$ replaced with $n_{m}^{'}$. Moreover,
\[
\mathrm{area}({\cal R}^{'})=\mathrm{area}({\cal R})-\sum _{n=n_{0}}^{n_{m}^{'}-1}\frac{1}{f(n)^{2t}}=\sum _{n=n_{m}^{'}}^{\infty}\frac{1}{f(n)^{2t}}.
\]
Hence the identity (\ref{17}) with $n_{0}$ replaced with $n_{m}^{'}$ is obtained. Finally, the heights of the rectangles in ${\cal R}^{'}$ are at most $1$. Therefore, by induction hypothesis, squares of sidelength $f(n)^{-t}$ with $n_{m}^{'}\leq n <n_{\max}$ can be packed  into $\bigcup _{R^{'}\in {\cal R}^{'}}R^{'}$. On the other hand, squares of sidelength $f(n)^{-t}$ for $n_{0}\leq n<n_{m}^{'}$ have already been packed into ${\cal R}\backslash {\cal R}^{'}$. Hence the induction is now completed. 
\end{proof}

\subsection{Proof of Theorem 1.1}
Suppose that $f$, $t$, $K$, $l$, $M$ and $N_{0}$ satisfy all the conditions in the statement of Theorem 1.1.  Let $S$ be a square of area $\sum _{n=n_{0}}^{\infty}f(n)^{-2t}$. By (\ref{21}), the sidelength of $S$ is at most $d_{2}^{\frac{1}{2}}\eta _{2}(t)^{\frac{1}{2}}\lambda _{2}(n_{0})^{\frac{1}{2}}f(n_{0})^{-t}$. Therefore,
\begin{align*}
\mathrm{perim}_{\delta}(S)&\leq (d_{2}^{\frac{1}{2}}\eta _{2}(t)^{\frac{1}{2}}\lambda _{2}(n_{0})^{\frac{1}{2}}f(n_{0})^{-t})^{1+\delta} =(d_{2}\eta _{2}(t)\lambda _{2}(n_{0}))^{\frac{1+\delta}{2}}f(n_{0})^{-(1+\delta)t}.
\end{align*}
On the other hand, by (\ref{20}), 
\[
\sum _{n=1}^{n_{0}-1}\frac{1}{f(n)^{t+\delta t}}\geq c_{1}\xi _{1}(t)\nu _{1}(n_{0})f(n_{0})^{-(1+\delta)t}.
\]
Hence the condition (\ref{18}) is satisfied if
\[
(d_{2}\eta _{2}(t)\lambda _{2}(n_{0}))^{\frac{1+\delta}{2}}f(n_{0})^{-(1+\delta)t}\leq c_{1}M^{-1+\frac{\delta}{2}}\xi _{1}(t)\nu _{1}(n_{0})f(n_{0})^{-(1+\delta)t},
\]
and the above inequality is equivalent to (\ref{33}).  Furthermore, the height of $S$ is at most $1$ provided that (\ref{34}) is satisfied. Consequently, if the functions and parameters satisfy these conditions, one can apply Proposition 2.4 to  ${\cal R}=S$ and see that  for any $n_{\max}>n_{0}\; (\geq N_{0})$, the squares of sidelength $f(n)^{-t}$ $(n_{0}\leq n <n_{\max})$ can be packed into $S$. Since $n_{\max}$ is arbitrary, by sending $n_{\max}\to \infty$ and using a standard compactness argument (see e.g., \cite{Ma}),  we see that there is a perfectly packing of squares of sidelength $f(n)^{-t}$ for $n \geq n_{0}$ into $S$, under the assumptions of Theorem 1.1. Thus the proof of Theorem 1.1 is completed.  \hspace{12.2cm} $\Box$


\section{Perfectly packing a square by squares of nearly AP-harmonic sidelength}
In this section, we apply Theorem 1.1 to the problem of perfectly packing a square by squares of sidelength $(qn+r)^{-t}$ for $q>r\geq 0$, $1/2<t<1$. 
\begin{cor}
Let $1/2<t<1$ and $q> r\geq 0$ be fixed constants and put $\delta =1-t$. Suppose that $M$ and $N_{0}$ are positive integers satisfying
\[
M\geq \max\left\{4^{\frac{1}{1-t}}\left( \frac{2(2t-1)}{1-t-\delta t} \right)^{\frac{2}{t(1-t)}}, \left(50\left(\frac{11}{10}\right)^{t(2-t)} \right)^{\frac{2}{1-t}}    \right\},
\]
\begin{align*}
N_{0} \geq \max \biggl\{ &4752M^{4}, \; \frac{2r}{q((264M^{2})^{\frac{1}{t}}-2)}, \;  (10q(2q)^{t}M)^{\frac{1}{1-t}}, \; 2^{\frac{1}{1-t-\delta t}}\frac{q+r}{q},  \\
 &  \quad  q^{-1}\left(\frac{q(1-t-\delta t)}{(q+r)^{t+\delta t}} \right)^{\frac{1}{1-t-\delta t}}, \; \frac{1}{1-2^{-\frac{1}{2t-1}}},  \\
  & q^{-1}(2q(1-t-\delta t))^{\frac{2}{1-\delta}}\left(\frac{2}{q(2t-1)} \right)^{\frac{1+\delta}{1-\delta}}M^{\frac{2-\delta}{1-\delta}}, \; q^{-1}\left( \frac{2}{q(2t-1)} \right)^{\frac{1}{2t-1}}  \biggr\}.
\end{align*}
Then, for any positive integer $n_{0}\geq N_{0}$, squares of sidelength $(qn+r)^{-t}$ for $n \geq n_{0}$ can be packed perfectly into a square of area $\sum _{n=n_{0}}^{\infty}(qn+r)^{-2t}$. 
\end{cor}
\begin{rem}[Effective lower bound for $N_{0}$ in the case of nearly harmonic sidelength]
In the case of $f(n)=n$ (hence $q=1, r=0$), a specific lower bound for $N_{0}$ is given by the following table.

\begin{table}[htbp]
\centering
\caption{Lower bound for $N_{0}$ in the case $f(n)=n$}
\begin{tabular}{|c||c|c|c|c|c|c|c|c|} \hline
$t$ &0.55  &0.6 &0.7  & 0.8 &0.9 &0.95  & 0.99 &0.999 \\ \hline
$\log _{10}N_{0}$ &35 &38 & 50 & 114 & 563 &2673 &92863 & 13216295  \\ \hline
\end{tabular}
\end{table}
\noindent
As $t\nearrow 1$, a sufficient condition for the value of $N_{0}$ is
\[
N_{0}\geq 4^{\frac{1}{(1-t)^{2}}}(2^{t}\cdot 10)^{\frac{1}{1-t}}\left( \frac{2(2t-1)}{(1-t)^{2}} \right)^{\frac{2}{t(1-t)^{2}}}.
\]
Roughly speaking, as $t\nearrow 1$, the perfectly packing is possible if $N_{0}\geq \left(3/(1-t)^{2}\right)^{3/(1-t)^{2}}$. 
\end{rem}

\begin{proof}
All we have to do is to identify when the function $f(x)=qx+r$ satisfies the whole conditions of  Theorem 1.1. The condition (\ref{15}) is
\[
q\leq \frac{1}{4752M^{4}}(qn_{0}+r),
\]
which is satisfied if
\begin{equation}
\label{A1}
n_{0}\geq 4752 M^{4}.
\end{equation}
The condition (\ref{15.5}) is
\[
(q(n_{0}+9M^{2})+r) \leq (264M^{2})^{\frac{1}{t}}(qn_{0}+r)
\]
If $n_{0}$ satisfies (\ref{A1}), then $n_{0}+9M^{2}\leq 2n_{0}$. Therefore, the left hand side of the above inequality is at most $2(qn_{0}+r)$. Hence a sufficient condition for (\ref{15.5}) is
\[
2(qn_{0}+r)\leq (264M^{2})^{\frac{1}{t}}(qn_{0}+r),
\]
which is satisfied when
\begin{equation}
\label{A2}
n_{0}\geq \frac{2r}{q((264M^{2})^{\frac{1}{t}}-2)}.
\end{equation}
The condition (\ref{16}) is 
\[
qn_{0}+r \geq K(q(1+l)n_{0}+r),
\]
which holds when
\begin{equation}
\label{A3}
K\leq \frac{1}{1+l}.
\end{equation}
We choose $l=1/10, \; K=10/11$. Then $K\geq (2/3)^{\frac{1}{2}}$ is also satisfied. The condition (\ref{31}) is
\[
9M(qn_{0}+r)^{t}\leq ln_{0}.
\]
The left hand side is at most $9(2q)^{t}Mn_{0}^{t}$. Therefore, (\ref{31}) holds for  
\begin{equation}
\label{A4}
n_{0}\geq \left( \frac{9(2q)^{t}M}{l}  \right)^{\frac{1}{1-t}}.
\end{equation}
Next, we consider the conditions (\ref{20}), (\ref{21}). Evaluating the summation by integrals, we have
\begin{equation}
\label{int1}
\int _{1}^{n_{1}}\frac{dx}{(qx+r)^{t+\delta t}} \leq \sum_{n=1}^{n_{1}-1}\frac{1}{(qn+r)^{t+\delta t}}\leq \frac{1}{(q+r)^{t+\delta t}}+ \int _{1}^{n_{1}-1}\frac{dx}{(qx+r)^{t+\delta t}}.
\end{equation}
The left hand side equals 
\[
\frac{1}{q(1-t-\delta t)}\frac{qn_{1}+r}{(qn_{1}+r)^{t+\delta t}}-\frac{1}{q(1-t-\delta t)}\frac{q+r}{(q+r)^{t+\delta t}},
\]
which is equal to or larger than
\[
\frac{1}{2}\frac{1}{q(1-t-\delta t)}\frac{qn_{1}+r}{(qn_{1}+r)^{t+\delta t}}
\]
if $n_{1}$ satisfies 
\begin{equation}
\label{A5}
n_{1}\geq 2^{\frac{1}{1-t-\delta t}}\frac{q+r}{q}.
\end{equation}
Therefore, for $n_{1}$ satisfying the condition (\ref{A5}), the former inequality of (\ref{20}) holds with 
\[
c_{1}=\frac{1}{2}, \quad \xi _{1}(t)=\frac{1}{q(1-t-\delta t)}, \quad \nu _{1}(n_{1})=qn_{1}+r.
\]
On the other hand, the right hand side of (\ref{int1}) is at most
\[
\frac{1}{(q+r)^{t+\delta t}}+\frac{1}{q(1-t-\delta t)}\frac{qn_{1}+r}{(qn_{1}+r)^{t+\delta t}}.
\]
This is equal to or less than 
\[
\frac{2}{q(1-t-\delta t)}\frac{qn_{1}+r}{(qn_{1}+r)^{t+\delta t}}
\]
if $n_{1}$ satisfies
\begin{equation}
\label{A6}
n_{1}\geq q^{-1}\left( \frac{q(1-t-\delta t)}{(q+r)^{t+\delta t}} \right)^{\frac{1}{1-t-\delta t}}.
\end{equation}
Therefore, for $n_{1}$ satisfying the condition (\ref{A6}), the latter inequality of (\ref{20}) holds with 
\[
c_{2}=2, \quad \xi _{2}(t)=\frac{1}{q(1-t-\delta t)}, \quad \nu _{2}(n_{1})=qn_{1}+r.
\]
Next, we consider (\ref{21}). We have
\[
\int _{n_{1}}^{\infty}\frac{dx}{(qx+r)^{2t}} \leq \sum _{n=n_{1}}^{\infty}\frac{1}{(qn+r)^{2t}} \leq \int _{n_{1}-1}^{\infty}\frac{dx}{(qx+r)^{2t}}.
\]
The left hand side equals
\[
\frac{qn_{1}+r}{q(2t-1)(qn_{1}+r)^{2t}}.
\]
Hence the former inequality of (\ref{21}) holds with 
\[
d_{1}=1, \quad \eta _{1}(t)=\frac{1}{q(2t-1)}, \quad  \lambda _{1}(n_{1})=qn_{1}+r.
\]
The right hand side is 
\[
\frac{1}{q(2t-1)}\frac{q(n_{1}-1)+r}{(q(n_{1}-1)+r)^{2t}},
\]
which becomes smaller than
\[
\frac{2}{q(2t-1)}\frac{qn_{1}+r}{(qn_{1}+r)^{2t}}
\]
when
\begin{equation}
\label{A7}
n_{1}\geq \frac{1}{1-2^{-\frac{1}{2t-1}}}.
\end{equation}
Therefore, if $n_{1}$ satisfies the condition (\ref{A7}), then the latter inequality of (\ref{21}) holds with
\[
d_{2}=2, \quad \eta _{2}(t)=\frac{1}{q(2t-1)}, \quad \lambda _{2}(n_{1})=qn_{1}+r.
\]
We assume (\ref{A5}), (\ref{A6}) and (\ref{A7}). Then, (\ref{25}) holds with $c_{\lambda ,\nu }=1$. The condition (\ref{33}) is
\[
(qn_{0}+r)^{1-\frac{1+\delta}{2}}\geq 2q(1-t-\delta t) \left(\frac{2}{q(2t-1)} \right)^{\frac{1+\delta}{2}}M^{1-\frac{\delta}{2}},
\]
which holds for 
\begin{equation}
\label{A8}
n_{0}\geq q^{-1}(2q(1-t-\delta t))^{\frac{2}{1-\delta}}\left(\frac{2}{q(2t-1)} \right)^{\frac{1+\delta}{1-\delta}}M^{\frac{2-\delta}{1-\delta}}.
\end{equation}
Finally, the condition (\ref{34}) is
\[
qn_{0}+r \leq \frac{1}{2}q(2t-1)(qn_{0}+r)^{2t},
\]
which holds for 
\begin{equation}
\label{A9}
n_{0}\geq q^{-1}\left( \frac{2}{q(2t-1)} \right)^{\frac{1}{2t-1}}.
\end{equation}
We choose a positive integer $N_{0}$ that is larger than  the largest value of the right hand sides of (\ref{A1}), (\ref{A2}), (\ref{A4}),  (\ref{A5}) to  (\ref{A9}). Then for $1/2<t<1$ and for any $n_{0}\geq N_{0}$, the squares of sidelength $(qn+r)^{-t}$ for $n\geq n_{0}$  are packed perfectly into a square of area $\sum _{n=n_{0}}^{\infty}(qn+r)^{-2t}$.
\end{proof}

\section{Perfectly packing a square by squares of nearly P-harmonic sidelength}
In this section, for $1/2<t<1$, we give a perfectly packing a square by squares of sidelength $p_{n}^{-t}$ for $n\geq n_{0}$ for any $n_{0}\geq N_{0}$, where $p_{n}$ denotes the $n$th prime number and $N_{0}$ is a sufficiently large positive integer depending on $t$, which will be explicitly given in Corollary 4.6. Before the statement of the theorem, we prepare some basic facts on primes. First, it is known that the sequence $\{p_{n} \}$ satisfies
\begin{equation}
\label{LP}
p_{n}>n\log n
\end{equation}
for any positive integer $n$, and 
\begin{equation}
\label{UP}
p_{n}<n(\log n +\log \log n)
\end{equation}
for any $n\geq 6$ (see for example  \cite{R}). We first give an upper bound for $\sum _{n \leq x}1/p_{n}^{t}$ for any fixed $1/2<t<1$.  By decomposing the interval $[1, x]$ into $[1, \sqrt[4]{x^{3}}]\cup (\sqrt[4]{x^{3}},x]$ and using (\ref{LP}), we have
\begin{equation}
\label{P1}
\begin{aligned}
\sum _{n\leq x}\frac{1}{p_{n}^{t}} \leq & \frac{1}{2^{t}}+\frac{1}{(\log 2)^{t}}\sum _{2\leq n \leq \sqrt[4]{x^{3}}}\frac{1}{n^{t}}+\frac{1}{(\log \sqrt[4]{x^{3}})^{t}}\sum _{\sqrt[4]{x^{3}}<n \leq x} \frac{1}{n^{t}} \\
\leq &  \frac{1}{2^{t}} +\frac{1}{(\log 2)^{t}}\int _{1}^{ \sqrt[4]{x^{3}}}\frac{du}{u^{t}}+\frac{\left( \frac{4}{3} \right)^{t}}{(\log x)^{t}} \int _{2}^{x}\frac{du}{u^{t}} \\
\leq & \frac{1}{2^{t}} +\frac{x^{\frac{3}{4}(1-t)}}{(1-t)(\log 2)^{t}}+\frac{\left(\frac{4}{3} \right)^{t} x^{1-t}}{(1-t)(\log x)^{t}}.
\end{aligned}
\end{equation}
In the right hand side of  (\ref{P1}), the inequality
\[
\max \left\{ \frac{1}{2^{t}}, \;  \frac{x^{\frac{3}{4}(1-t)}}{(1-t)(\log 2)^{t}} \right\}    \leq \frac{ x^{1-t}}{(1-t)(\log x)^{t}}
\]
holds for 
\begin{equation}
\label{P3}
x \geq e^{\frac{16}{(1-t)^{2}}}.
\end{equation}
Furthermore,
\[
\frac{\left(\frac{4}{3} \right)^{t} x^{1-t}}{(1-t)(\log x)^{t}} \leq \frac{4}{3}\frac{ x^{1-t}}{(1-t)(\log x)^{t}}
\]
for any $x>1$. Therefore, by (\ref{P1}) and the above estimates, we have
\begin{equation}
\label{P4}
\sum _{n\leq x}\frac{1}{p_{n}^{t}} \leq \frac{10}{3} \frac{ x^{1-t}}{(1-t)(\log x)^{t}}
\end{equation}
for $x$ satisfying (\ref{P3}). On the other hand, since $1/(1+(\log \log n/\log n)) >7/10$ for any $n \geq 6$, by (\ref{UP}) we have
\begin{equation}
\label{P2}
\begin{aligned}
\sum _{n\leq x}\frac{1}{p_{n}^{t}} &>\frac{7}{10}\sum _{6\leq n \leq x}\frac{1}{(n\log n)^{t}}  \geq \frac{7}{10}\frac{1}{(\log x)^{t}}\sum _{6\leq n \leq x}\frac{1}{n^{t}}  \geq \frac{7}{10}\frac{1}{(\log x)^{t}}\int _{6}^{x}\frac{du}{u^{t}} \\
& \geq \frac{7(x^{1-t}-6^{1-t})}{10(1-t)(\log x)^{t}}
\end{aligned}
\end{equation}
for $x \geq 6$. If $x$ satisfies 
\begin{equation}
\label{P4.5}
x\geq 6\cdot 2^{\frac{1}{1-t}},
\end{equation}
then $x^{1-t}-6^{1-t}\geq  x^{1-t}/2$. Therefore, for $x$ satisfying the condition (\ref{P4.5}), we have 
\begin{equation}
\label{P5}
\sum _{n\leq x}\frac{1}{p_{n}^{t}}\geq \frac{7}{20}\frac{x^{1-t}}{(1-t)(\log x)^{t}}.
\end{equation}
Finally, we note that $e^{\frac{16}{(1-t)^{2}}}>6\cdot 2^{\frac{1}{1-t}}$ holds for any $1/2<t<1$, we have the following conclusion.
\begin{lem}
For any $1/2<t<1$ and $x\geq e^{\frac{16}{(1-t)^{2}}}$, we have
\begin{equation}
\label{P6}
\frac{7}{20}\frac{x^{1-t}}{(1-t)(\log x)^{t}}<  \sum _{n\leq x}\frac{1}{p_{n}^{t}} < \frac{10}{3}\frac{x^{1-t}}{(1-t)(\log x)^{t}}. 
\end{equation}
\end{lem}
Replacing $x$ with $x-1$ in (\ref{P6}), we see that
\[
\frac{7}{20}\frac{(x-1)^{1-t}}{(1-t)(\log (x-1))^{t}}<  \sum _{n\leq x-1}\frac{1}{p_{n}^{t}} < \frac{10}{3}\frac{(x-1)^{1-t}}{(1-t)(\log (x-1))^{t}}
\]
for $x\geq 1+e^{\frac{16}{(1-t)^{2}}}$.  In this region it follows that 
\[
\frac{1}{2}\frac{x^{1-t}}{(1-t)(\log x)^{t}}< \frac{(x-1)^{1-t}}{(1-t)(\log (x-1))^{t}} <\frac{2x^{1-t}}{(1-t)(\log x)^{t}}.
\]
Combining (\ref{P6}) and the above inequality, we have the following result.
\begin{lem}
For $1/2<t<1$ and $x\geq 2e^{\frac{16}{(1-t)^{2}}}$, we have
\begin{equation}
\label{P10}
\frac{7}{40} \frac{x^{1-t}}{(1-t)(\log x)^{t}} < \sum _{n\leq x-1}\frac{1}{p_{n}^{t}} < \frac{20}{3}  \frac{x^{1-t}}{(1-t)(\log x)^{t}}.
\end{equation}
\end{lem}
Next, we evaluate $\sum _{n \geq x}1/p_{n}^{t}$ for $t>1$. By (\ref{LP}), we have
\[
\sum _{n\geq x}\frac{1}{p_{n}^{t}}<\frac{1}{(\log x)^{t}}\sum _{n\geq x}\frac{1}{n^{t}}\leq \frac{1}{(\log x)^{t}}\int _{x-1}^{\infty}\frac{du}{u^{t}}=\frac{(x-1)^{1-t}}{(t-1)(\log x)^{t}}.
\]
If $x\geq 1/(1-2^{-\frac{1}{t-1}})$, then $(x-1)^{1-t}\leq 2x^{1-t}$. Therefore, for this $x$ we have
\begin{equation}
\label{P8}
\sum _{n\geq x}\frac{1}{p_{n}^{t}}\leq \frac{2x^{1-t}}{(t-1)(\log x)^{t}}.
\end{equation}
On the other hand, by (\ref{UP}), for $x\geq 6$ we have
\begin{align*}
\sum _{n\geq x}\frac{1}{p_{n}^{t}}&\geq \sum _{n\geq x}\frac{1}{(n\log n)^{t}\left(1+\frac{\log \log n}{\log n} \right)^{t}} \geq \frac{1}{2^{t}}\sum _{n\geq x}\frac{1}{(n\log n)^{t}}  \geq \frac{1}{2^{t}}\frac{1}{(\log 2x)^{t}}\sum _{x\leq n <2x}\frac{1}{n^{t}} \\
& \geq \frac{1}{(2\log 2x)^{t}}\int _{x}^{2x}\frac{du}{u^{t}}  = \frac{(1-2^{1-t})x^{1-t}}{(t-1)(2\log 2x)^{t}}  \geq \frac{(1-2^{1-t})x^{1-t}}{2^{2t}(t-1)(\log x)^{t}}.
\end{align*}
Consequently,
\begin{lem}
For $t>1$ and $x\geq \max \{1/(1-2^{-\frac{1}{t-1}}), 6 \}$, we have
\begin{equation}
\label{P9}
\frac{(1-2^{1-t})x^{1-t}}{2^{2t}(t-1)(\log x)^{t}}\leq \sum _{n\geq x}\frac{1}{p_{n}^{t}}\leq  \frac{2x^{1-t}}{(t-1)(\log x)^{t}}.
\end{equation}
\end{lem}
Finally, by (\ref{LP}) and (\ref{UP}), it follows that
\[
\frac{p_{n_{1}}}{2}<n_{1}\log n_{1}<p_{n_{1}}.
\]
Combining the above inequality and (\ref{P10}), (\ref{P9}), we obtain the following conclusion. 
\begin{lem}
We have
\begin{equation}
\label{P11}
\frac{7}{40}\frac{n_{1}}{(1-t)p_{n_{1}}^{t}}<\sum _{n\leq n_{1}-1}\frac{1}{p_{n}^{t}}<\frac{20}{3}\frac{2^{t}n_{1}}{(1-t)p_{n_{1}}^{t}}
\end{equation}
for $1/2<t<1$, $n_{1}\geq 2e^{\frac{16}{(1-t)^{2}}}$, and 
\begin{equation}
\label{P12}
\frac{(1-2^{1-t})n_{1}}{2^{2t}(t-1)p_{n_{1}}^{t}}<\sum _{n\geq n_{1}}\frac{1}{p_{n_{1}}^{t}}<\frac{2^{1+t}n_{1}}{(t-1)p_{n_{1}}^{t}}
\end{equation}
for $t>1$, $n_{1}\geq \max \{6, 1/(1-2^{-\frac{1}{t-1}}) \}$. 
\end{lem}
The next purpose is to extend the sequence $\{p_{n} \}$ to a smooth increasing function on $\mathbb{R}_{\geq 1}$.
\begin{lem}
There exists a smooth increasing function $f: \mathbb{R}_{\geq 1}\to \mathbb{R}_{\geq 1}$ which satisfies the following two conditions. \\
1. For any positive integer $n$, $f(n)=p_{n}$. \\
2. There exists an absolute positive constant $C_{P}<7$ such that for any positive integer $n$ and for any $x\in [n, n+1]$, 
\begin{equation}
\label{A10}
f^{'}(x) \leq C_{p}(p_{n+1}-p_{n}).
\end{equation}
\begin{proof}
Define a function $\varphi$ by
\[\varphi (x)=
\begin{cases}
0 & (x<-\frac{1}{6}) \\
\frac{1}{2}\exp \left(1-\frac{1}{1-36x^{2}} \right) &(-\frac{1}{6}\leq x \leq 0 ) \\
1-\frac{1}{2}\exp \left(1-\frac{1}{1-36x^{2}} \right)  &(0\leq x \leq \frac{1}{6}) \\
1 & (x>\frac{1}{6}).
\end{cases}
\]
\
This is a smooth increasing function satisfying $\varphi (x)=0$ if $x<-1/6$,  $\varphi (x)=1$ if $x>1/6$ and $\max \varphi ^{'}(x)=6.511\ldots <7$. We define the function $f(x)$ by
\begin{equation}
\label{f}
f(x)=p_{1}+\sum _{k=1}^{\infty}\varphi (x-(k+1/2))(p_{k+1}-p_{k}).
\end{equation}
We fix a positive integer $n$ arbitrarily and suppose $x\in [n, n+1]$. If $k\geq n+1$, then
\[
x-\left(k+\frac{1}{2} \right)\leq (n+1)-(n+1)-\frac{1}{2}=-\frac{1}{2}<-\frac{1}{6}
\]
Hence $\varphi (x-(k+1/2))=0$. If $k\leq n-1$, then
\[
x-\left(k+\frac{1}{2}\right)\geq n-(n-1)-\frac{1}{2}=\frac{1}{2}>\frac{1}{6}.
\]
Hence $\varphi (x-(k+1/2))=1$. Therefore,
\begin{align*}
f(x)=&p_{1}+\sum _{k=1}^{n-1}(p_{k+1}-p_{k})+\varphi (x-(n+1/2))(p_{n+1}-p_{n}) \\
=& p_{n}+\varphi (x-(n+1/2))(p_{n+1}-p_{n}).
\end{align*}
Hence $f$ satisfies the conditions of the lemma. 
\end{proof}
\end{lem}
Next, it is known that there exist some $0<\theta <1$ and a positive integer $N_{\theta}$ for which
\begin{equation}
\label{A11}
p_{n+1}-p_{n}\leq p_{n}^{\theta}
\end{equation}
holds for any $n\geq N_{\theta}$. The best result currently known is due to Baker, Harman and Pintz \cite{BHP}, according to which (\ref{A11}) holds for $\theta =0.525$. By (\ref{A10}) and (\ref{A11}), we have
\[
f^{'}(x)\leq C_{P}p_{n}^{\theta}
\]
for any $n\geq N_{\theta}$, $x \in [n, n+1]$. Combining this with $p_{n}<2n\log n$ $(n\geq 6)$, we have
\begin{equation}
\label{A11.5}
f^{'}(x)\leq 7(2\log n)^{\theta}
\end{equation}
for any $n\geq \max\{ 6, N_{\theta} \}$, $x \in [n, n+1]$. Now our theorem on perfectly packing a square by squares of sidelength $p_{n}^{-t}$ is described as follows.
\begin{cor}
Assume that  (\ref{A11}) holds for any $n\geq N_{\theta}$ with some positive constant $\theta <1$.  for $1/2<t<1$, let $M$ be a positive integer satisfying
\[
M\geq \max \left\{  4^{\frac{1}{1-t}} \left(\frac{20\cdot 2^{6t-t^{2}}(2t-1)}{3(1-t)^{2}(1-2^{1-2t})}\right)^{\frac{2}{t(1-t)}}, \left(50\left( \frac{6}{5} \right)^{t(2-t)} \right)^{\frac{2}{1-t}}   \right\},
\]
and $N_{0}$ be a positive integer satisfying 
\begin{align*}
N_{0}\geq & \max \biggl\{N_{\theta},\;  85536 M^{5},  \; 3\times 10^{18}, \; 2e^{\frac{16}{(1-t)^{4}}}, \;  \frac{1}{1-2^{-\frac{1}{2t-1}}}, \\
& \quad \quad\quad  M^{\frac{1}{(1-t)^{2}}}, \;  \left( \frac{40}{7}(1-t)^{2} \right)^{\frac{2}{t}}\left( \frac{2^{1+2t}}{2t-1} \right)^{\frac{2-t}{t}}M^{\frac{1+t}{t}}, \; e^{\frac{1}{(2t-1)(1-t)}} \biggr\}.
\end{align*}
Then, for any positive integer $n_{0}\geq N_{0}$, there exists a perfectly packing a square of area $\sum _{n=n_{0}}^{\infty}p_{n}^{-2t}$ by squares of sidelength $p_{n}^{-t}$ for $n \geq n_{0}$. 
\end{cor}
\begin{rem}
Roughly speaking, the perfectly packing exists as $t\nearrow 1$ if $N_{0}\geq 2\exp \left( \frac{16}{(1-t)^{4}} \right)$. 
\end{rem}
\begin{proof}
From now on, we assume 
\begin{equation}
\label{A12}
N_{0}\geq N_{\theta}.
\end{equation}
By (\ref{A11.5}), it follows that
\[
\max _{x\in [n_{0}, n_{0}+9M^{2}]}f^{'}(x)\leq 8(2\log (n_{0}+9M^{2}))^{\theta}
\]
for any $n_{0}\geq N_{\theta}$, the condition (\ref{15}) is valid if
\[
 8(2\log (n_{0}+9M^{2}))^{\theta} \leq \frac{1}{4752}M^{-4}n_{0}\log n_{0} \quad (\forall n_{0}\geq N_{0}),
\]
which holds for 
\begin{equation}
\label{A13}
N_{0}\geq 2\cdot 4752 \cdot 9 M^{5}.
\end{equation}
Since $n\log n <f(n)<2n\log n$ for any $n\geq 6$, the condition (\ref{15.5}) is valid if 
\[
2(n_{0}+9M^{2})\log (n_{0}+9M^{2})\leq (264 M^{2})^{\frac{1}{t}}n_{0}\log n_{0}
\]
holds for any $n_{0}\geq N_{0}$. Moreover, if
\begin{equation}
\label{A14}
N_{0}\geq 9M^{2},
\end{equation}
then the left hand side is at most $4n_{0}\log 2n_{0}$. Hence under the assumption of (\ref{A14}), the condition (\ref{15.5}) is satisfied if
\[
4n_{0} \log n_{0}\leq  (264 M^{2})^{\frac{1}{t}}n_{0} \log n_{0},
\]
and this inequality holds whenever $M$ satisfies (\ref{M}). Consequently, if $M$ satisfies (\ref{M}) and $N_{0}$ satisfies (\ref{A14}), then the condition (\ref{15.5}) is also satisfied. The condition (\ref{16}) is valid if
\[
n_{0}\log n_{0}\geq K(1+l)n_{0}(\log (1+l)n_{0}+\log \log (1+l)n_{0})
\]
holds for any $n_{0}\geq N_{0}$, and by choosing
\[
l=\frac{1}{10}, \; K=\frac{5}{6},
\]
the above inequality holds for 
\begin{equation}
\label{A15}
N_{0}\geq 3\times 10^{18}.
\end{equation}
By replacing the parameter $t$ in (\ref{P11}) and (\ref{P12}) with $t+\delta t$ $(\delta :=1-t)$, $2t$, respectively, the conditions (\ref{20}), (\ref{21}) hold with
\[
c_{1}=\frac{7}{40},\quad  c_{2}=\frac{20}{3},\quad  \xi _{1}(t)=\frac{1}{(1-t)^{2}},\quad \xi _{2}(t)=\frac{2^{2t-t^{2}}}{(1-t)^{2}}, \quad   \nu _{1}(n_{1})=\nu _{2}(n_{1})=n_{1},
\]
\[d_{1}=d_{2}=1, \quad \eta _{1}(t)=\frac{1-2^{1-2t}}{2^{4t}(2t-1)}, \quad \eta _{2}(t)=\frac{2^{1+2t}}{2t-1}, \quad  \lambda _{1}(n_{1})=\lambda _{2}(n_{1})=n_{1}.
\]
(We used $1-t-\delta t=(1-t)^{2}$.) Hence (\ref{25}) holds with $c_{\lambda ,\nu}=1$. The condition (\ref{31}) holds if
\[
9Mn_{0}^{t}(\log n_{0}+\log \log n_{0})^{t}\leq \frac{1}{10}n_{0} \quad (\forall n_{0}\geq N_{0}),
\]
which is equivalent to
\[
n_{0}^{1-t}\geq 90M(\log n_{0}+\log \log n_{0})^{t} \quad (\forall n_{0}\geq N_{0}).
\]
Since the right hand side is at most $180M\log n_{0}$, a sufficient condition for (\ref{31}) is that
\[
n_{0}^{1-t}\geq 180M\log n_{0} \quad (\forall n_{0}\geq N_{0}).
\]
This condition is satisfied if
\begin{equation}
\label{A16}
N_{0}\geq M^{\frac{1}{(1-t)^{2}}}.
\end{equation}
The condition (\ref{33}) is
\[
n_{0}^{1-\frac{1+\delta}{2}}\geq \frac{\left( \frac{2^{1+2t}}{2t-1} \right)^{\frac{1+\delta}{2}}}{\frac{7}{40}\frac{1}{(1-t)^{2}}}M^{1-\frac{\delta}{2}} \quad (\forall n_{0}\geq N_{0}),
\]
which is satisfied if
\begin{equation}
\label{A17}
N_{0}\geq \left( \frac{40}{7}(1-t)^{2} \right)^{\frac{2}{t}}\left( \frac{2^{1+2t}}{2t-1} \right)^{\frac{2-t}{t}}M^{\frac{1+t}{t}}.
\end{equation}
Finally, the condition (\ref{34}) is satisfied if
\[
n_{0}\leq \frac{2t-1}{2^{1+2t}}(n_{0}\log n_{0})^{2t} \quad (\forall n_{0}\geq N_{0}),
\]
and this inequality holds for 
\begin{equation}
\label{A18}
N_{0}\geq e^{\frac{1}{(2t-1)(1-t)}}.
\end{equation}
Combining the lower bounds for $N_{0}$, we obtain the conclusion.
\end{proof}
\section{Ineffective version of the perfectly packing theorem and an application to the packing of squares of nearly TP-harmonic sidelength}
\subsection{Ineffective version of the perfectly packing theorem}
Clearly, one cannot adapt Theorem 1.1 when $f(x)$ is larger than $x^{1+\delta _{0}}$ for some constant $\delta _{0}>0$ up to constant multiples, because the condition (\ref{31}) fails. On the other hand, we have seen that one can adapt Theorem 1.1 to the function with  $f(n)=p_{n}$, where $p_{n}\sim n\log n$. Hence one might wonder that how large the function $f(x)$ can be. The following theorem partially answers this question.
\begin{thm}
Let $\kappa : \mathbb{R}_{>1}\to \mathbb{R}_{>1}$ be a smooth function for which $f(x):=x\kappa (x)$ becomes a smooth increasing function. Suppose that for any $\varepsilon >0$, there exists a positive integer $N_{\varepsilon}$ such that for any $n\geq N_{\varepsilon}$, 
\[
\max _{x \in [n, 2n]}\kappa ^{'}(x) \leq \varepsilon \kappa (x),
\]
\[
(1-\varepsilon )\kappa (M(2n))\leq \kappa (n) \leq n^{\varepsilon},
\]
\[
c_{\mathfrak{m}}\kappa (n)\leq \kappa (\mathfrak{m}(n))\leq C_{\mathfrak{m}}\kappa (n),
\]
where $M(2n)$ denotes one of the values $x\in [1,2n]$ such that $\kappa (x)$ takes its minimum in $[1,2n]$,  $\mathfrak{m}(n)$ denotes one of the values $x\in [n, \infty)$ such that $\kappa (x)$ takes its minimum in $[n, \infty )$, and $C_{\mathfrak{m}}>c_{\mathfrak{m}}>0$ are positive constants. Moreover, we assume that there exist constants $0<\alpha <1$ and $C_{\alpha}>1$ such that 
\[
\kappa (n)\leq C_{\alpha}\kappa (m_{\alpha}(n))
\]
holds for any $n\geq N_{\varepsilon}$, where $m_{\alpha}(n)$ denotes one of the values $x\in [n, \infty )$ such that $\kappa (x)$ takes its minimum in $[n, \infty )$. Then, for any  $1/2<t<1$, there exists a positive integer $N_{0}=N_{0}(t)$ such that for any integer $n_{0}\geq N_{0}$,  squares of sidelength $(n\kappa (n))^{-t}$ for $n \geq n_{0}$ can be packed perfectly into a square of area $\sum _{n=n_{0}}^{\infty}(n \kappa (n))^{-2t}$. 
\end{thm}
\begin{proof}
In this proof we assume that $M$ satisfies (\ref{M}) and that $N_{0}\geq 9M^{2}$. The condition (\ref{15}) becomes
\[
\max _{x\in [n_{0}, 2n_{0}]}(\kappa (x)+x\kappa ^{'}(x)) \leq \frac{1}{4752}M^{-4}n_{0}\kappa (n_{0}).
\]
This inequality holds if both 
\begin{equation}
\label{A19}
\kappa (M(2n_{0}))\leq \frac{1}{9504}M^{-4}n_{0}\kappa (n_{0})
\end{equation}
and 
\begin{equation}
\label{A20}
\max _{x\in [n_{0}, 2n_{0}]}\kappa ^{'}(x) \leq \frac{1}{19008}M^{-4}\kappa (n_{0})
\end{equation}
are satisfied. These two inequalities are valid if the first two conditions in the statement of the theorem are satisfied.  The condition (\ref{15.5}) is satisfied if
\[
2n_{0}\kappa (M(2n_{0})) \leq (264M^{2})^{\frac{1}{t}}(n_{0}\kappa (n_{0})),
\]
which is equivalent to
\[
\kappa (M(2n_{0}))\leq \frac{1}{2}(264M^{2})^{\frac{1}{t}}\kappa (n_{0}).
\]
The right hand side is at least $(1/2)(264M^{2})^{\frac{1}{2}}\kappa (n_{0})$. Therefore, (\ref{15.5}) holds if
\begin{equation}
\label{A21}
\kappa (M(2n_{0}))\leq \sqrt{66}M \kappa (n_{0}).
\end{equation}
The condition (\ref{16}) is
\[
n_{0}\kappa (n_{0})\geq K(1+l)n_{0}\kappa ((1+l)n_{0}).
\]
Since the right hand side is at most $K(1+l)n_{0}\kappa (M(2n_{0}))$ for $l<1$, the condition (\ref{16}) is valid if
\[
n_{0}\kappa (n_{0})\geq K(1+l)n_{0}\kappa (M(2n_{0})).
\]
By taking $l=\varepsilon$, $K=(1-\varepsilon)/(1+\varepsilon)$, the above condition becomes
\begin{equation}
\label{A22}
\kappa (n_{0})\geq (1-\varepsilon)\kappa (M(2n_{0})).
\end{equation}
Notice that if we assume (\ref{A22}), then (\ref{A19}) and (\ref{A21}) automatically hold for any sufficiently large $n_{0}$. The condition (\ref{31}) is
\[
9M(n_{0}\kappa (n_{0}))^{t}\leq ln_{0}.
\]
This condition is valid if
\begin{equation}
\label{A23}
\kappa (n_{0})\ll n_{0}^{\varepsilon}
\end{equation}
holds for any $\varepsilon >0$, where the implied constant above is dependent only on $\varepsilon$. Next, we consider the condition (\ref{20}). Since
\[
\frac{1}{f(n)^{t+\delta t}}\geq \frac{1}{n^{t+\delta t}\kappa (M(n_{1}))^{t+\delta t}}
\]
for $1\leq n \leq n_{1}-1$, we have
\begin{align*}
\sum _{n=1}^{n_{1}-1}\frac{1}{f(n)^{t+\delta t}}\geq & \frac{1}{\kappa (M(n_{1}))^{t+\delta t}}\sum _{n=1}^{n_{1}-1}\frac{1}{n^{t+\delta t}} \geq  \frac{1}{\kappa (M(n_{1}))^{t+\delta t}} \int _{1}^{n_{1}}\frac{du}{u^{t+\delta t}} \\
= & \frac{n_{1}^{1-t-\delta t}-1}{(1-t-\delta t)\kappa (M(n_{1}))^{t+\delta t}} \geq  \frac{(1-\varepsilon )n_{1}}{2(1-t-\delta t)(\kappa (n_{1})n_{1})^{t+\delta t}}
\end{align*}
for any sufficiently large $n_{1}$. On the other hand, for any fixed $0<\alpha <1$ and for any sufficiently large $n_{1}$, we have
\begin{align*}
\sum _{n=1}^{n_{1}-1}\frac{1}{f(n)^{t+\delta t}}&\leq \sum _{1\leq n\leq n_{1}^{\alpha}}\frac{1}{\kappa (n)^{t+\delta t}n^{t+\delta t}} +\sum _{n_{1}^{\alpha} \leq n \leq n_{1}}\frac{1}{\kappa (n)^{t+\delta t}n^{t+\delta t}}  \\
&\leq \sum _{1\leq n\leq n_{1}^{\alpha}} \frac{1}{n^{t+\delta t}}   +\frac{1}{\kappa (m_{\alpha} (n_{1}))^{t+\delta t}} \sum _{n=1}^{n_{1}}\frac{1}{n^{t+\delta t}} \\
&\leq \frac{n_{1}^{\alpha (1-t-\delta t)}}{1-t-\delta t}+\frac{n_{1}^{1-t-\delta t}}{(1-t-\delta t)\kappa (m_{\alpha}(n_{1}))^{t+\delta t}} \\
&\leq \frac{2n_{1}^{1-t-\delta t}}{(1-t-\delta t)\kappa (m_{\alpha}(n_{1}))^{t+\delta t}} =\frac{2n_{1}\left( \frac{\kappa (n_{1})}{\kappa (m_{\alpha}(n_{1}))} \right)^{t+\delta t}}{(1-t-\delta t)(n_{1}\kappa (n_{1}))^{t+\delta t}},
\end{align*}
provided that
\[
\kappa (m_{\alpha}(n_{1}))^{t+\delta t}\leq n_{1}^{(1-\alpha)(1-t-\delta t)}.
\]
The above condition is satisfied if $\kappa$ satisfies $\kappa (n)\ll n^{\varepsilon}$ for any $\varepsilon >0$.  Combining these inequalities, for any sufficiently large $n_{1}$ and  any fixed $0<\alpha <1$, we have
\[
\frac{(1-\varepsilon )n_{1}}{2(1-t-\delta t)(\kappa (n_{1})n_{1})^{t+\delta t}} \leq \sum _{n=1}^{n_{1}-1}\frac{1}{f(n)^{t+\delta t}} \leq \frac{2n_{1}\left( \frac{\kappa (n_{1})}{\kappa (m_{\alpha}(n_{1}))} \right)^{t+\delta t}}{(1-t-\delta t)(n_{1}\kappa (n_{1}))^{t+\delta t}}.
\]
Therefore, the condition (\ref{20}) holds with 
\[
c_{1}=\frac{1-\varepsilon}{2}, \quad \xi _{1}(t)=\frac{1}{1-t-\delta t}, \quad \nu _{1}(n_{1})=n_{1},
\]
\[
c_{2}=2, \quad \xi _{2}(t)=\frac{1}{1-t-\delta t}, \quad \nu _{2}(n_{1})=\frac{n_{1}\kappa (n_{1})}{\kappa (m_{\alpha}(n_{1}))}.
\]
Next, we consider the condition (\ref{21}). For $t>1/2$, 
\begin{align*}
\sum _{n=n_{1}}^{\infty}\frac{1}{f(n)^{2t}}&\geq \frac{1}{\kappa (M(2n_{1}))^{2t}}\sum _{n=n_{1}}^{2n_{1}}\frac{1}{n^{2t}}  \geq \frac{1}{\kappa (M(2n_{1}))^{2t}} \int _{n_{1}}^{2n_{1}+1}\frac{du}{u^{2t}} \\
&\quad =\frac{n_{1}^{1-2t}-(2n_{1}+1)^{1-2t}}{(2t-1)\kappa (M(2n_{1}))^{2t}}  \geq \frac{n_{1}}{2(2t-1)\kappa (M(2n_{1}))^{2t}n_{1}^{2t}} \\
&\quad \quad \geq \frac{n_{1}}{2(1-\varepsilon)(2t-1)(\kappa (n_{1})n_{1})^{2t}},
\end{align*}
provided that (\ref{A22}) holds. On the other hand, 
\begin{align*}
\sum _{n=n_{1}}^{\infty}\frac{1}{f(n)^{2t}}\leq \frac{1}{\mathfrak{m}(n_{1})^{2t}}\sum _{n=n_{1}}^{\infty}\frac{1}{n^{2t}} =\frac{n_{1}\left( \frac{\kappa (n_{1})}{\kappa (\mathfrak{m}(n_{1}))} \right)^{2t}}{(2t-1)(n_{1}\kappa (n_{1}))^{2t}}.
\end{align*}
Combining these inequalities, it follows that (\ref{21}) holds with
\[
d_{1}=\frac{1}{2(1-\varepsilon )}, \quad \eta _{1}(t)=\frac{1}{2t-1}, \quad \lambda _{1}(n_{1})=n_{1},
\]
\[
d_{2}=1, \quad \eta _{2}(t)=\frac{1}{2t-1}, \quad \lambda _{2}(n_{1})=n_{1}\left( \frac{\kappa (n_{1})}{\kappa (\mathfrak{m}(n_{1}))} \right)^{2}.
\]
The condition (\ref{25}) is satisfied if $\kappa$ satisfies
\begin{equation}
\label{A25}
\kappa (n_{0})\leq C_{\alpha}\kappa (m_{\alpha}(n_{0}))
\end{equation}
for some $0<\alpha<1$, $C_{\alpha}>0$. The condition (\ref{33}) becomes 
\[
n_{0}^{\frac{1-\delta}{2}}\left( \frac{\kappa (n_{0})}{\kappa (\mathfrak{m} (n_{0})) }\right)^{-(1+\delta)}\geq \frac{2(1-t-\delta t)}{1-\varepsilon}\left( \frac{1}{2t-1} \right)^{\frac{1+\delta}{2}} M^{1-\frac{\delta}{2}}.
\]
If $\kappa$ satisfies
\begin{equation}
\label{A26}
\kappa (\mathfrak{m}(n_{0}))\leq C_{\mathfrak{m}}\kappa (n_{0}),
\end{equation}
then the above inequality holds for any sufficiently large $n_{0}$. The condition (\ref{34}) becomes 
\[
n_{0}\left( \frac{\kappa (n_{0})}{\kappa (\mathfrak{m}(n_{0}))} \right)^{2}\leq (2t-1)(n_{0}\kappa (n_{0}))^{2t}.
\]
If $\kappa$ satisfies
\begin{equation}
\label{A27}
\kappa (\mathfrak{m}(n_{0}))\geq c_{\mathfrak{m}}\kappa (n_{0}),
\end{equation}
then the above inequality holds for any sufficiently large $n_{0}$. Combining these results we obtain the conclusion of the theorem.
\end{proof}

\subsection{Packing of squares of nearly TP-harmonic sidelength}
A prime number $p$ is called a {\it twin prime} if at least one of $p-2$ or $p+2$ is also a prime number.  We denote the $n$th twin prime by $\mathfrak{p}_{n}$. Then $\mathfrak{p}_{1}=3,\; \mathfrak{p}_{2}=5,\; \mathfrak{p}_{3}=7,\; \mathfrak{p}_{4}=11,\; \mathfrak{p}_{5}=13, \ldots$. For $x\geq 2$, we denote the number of twin primes below $x$ by $\pi _{2}(x)$. Then, it is known that for any sufficiently large $x$,
\begin{equation}
\label{TP}
\pi _{2}(x) \leq \frac{C \Pi _{2}x}{(\log x)^{2}}\left(1+c\frac{\log \log x}{\log x} \right),
\end{equation}
where 
\begin{equation}
\label{tpc}
\Pi _{2}=\underset{p\geq 3}{\prod _{p:\mathrm{prime}}}\left(1-\frac{1}{(p-1)^{2}} \right)=0.6601618\ldots
\end{equation}
is the {\it twin prime constant}, and  $c$ is some absolute constant. The coefficient $C$ is also some absolute constant. Though the value of $C$ has not been known correctly, according to the Hardy-Littlewood conjecture \cite{HL}, the correct value of $C$ is expected to be 2. Currently the best known upper bound is $C \leq 6.8325$ by Haugland \cite{H}. 
\begin{cor}
Let ${\mathfrak p}_{n}$ be the $n$th twin prime, $\Pi _{2}$ be the twin prime constant defined by (\ref{tpc}) and $C$ be a constant in (\ref{TP}). Then, for any $1/2<t<1$ and $C ^{'}>C$, there exists a positive integer $N_{0}$ depending on $t$ and $C ^{'}$ that for any $n_{0} \geq N_{0}$, squares of sidelength ${\mathfrak p}_{n}^{-t}$ for $n \geq n_{0}$ can be packed into a square of area $(C ^{'}\Pi _{2})^{2t}\sum _{n=n_{0}}^{\infty}(n(\log n)^{2})^{-2t}$.
\end{cor}
\begin{proof}
Since $\pi _{2}(\mathfrak{p}_{n})=n$, the estimate (\ref{TP}) yields 
\[
n\leq \frac{C \Pi _{2}\mathfrak{p}_{n}}{(\log \mathfrak{p}_{n})^{2}}\left(1+ c \frac{\log \log \mathfrak{p}_{n}}{\log \mathfrak{p}_{n}} \right).
\]
From this, we easily see that for any $C ^{'}>C$, 
\begin{equation}
\label{tpestimate}
\mathfrak{p}_{n}\geq \frac{n}{C ^{'}\Pi _{2}}(\log n)^{2}
\end{equation}
holds for any $n\geq N(C^{'})$, where $N(C^{'})$ is some positive integer depending on $C ^{'}$. We put
\[
f(x)=\frac{x}{C ^{'}\Pi _{2}}(\log x)^{2}, \quad \kappa (x)=\frac{1}{C ^{'}\Pi _{2}}(\log x)^{2}
\]
for $x>1$. By appropriately modifying  the value of $\kappa (x)$ when $x$ is small, $\kappa$ becomes a smooth function from $\mathbb{R}_{>1}$ to $\mathbb{R}_{>1}$ and satisfies all the conditions of Theorem 5.1. Therefore, for any $1/2<t<1$, there exists a positive integer $N_{0}(t)$ such that for any positive integer $n_{0}\geq N_{0}(t)$, squares of sidelength $((n/C ^{'}\Pi _{2})(\log n)^{2})^{-t}$ for $n \geq n_{0}$ can be packed perfectly into a square of area $\sum _{n=n_{0}}^{\infty}((n/C ^{'}\Pi _{2})(\log n)^{2})^{-2t}$. By (\ref{tpestimate}),  a square of sidelength $\mathfrak{p}_{n}^{-t}$ for $n\geq N(C^{'})$ is contained in a square of sidelength $((n/C ^{'}\Pi _{2})(\log n)^{2})^{-t}$. Hence by taking $N_{0}:=\max \{N_{0}(t), N(C^{'}) \}$,  we obtain the result.
\end{proof}


\section{Acknowledgements}
This work is partially supported by the JSPS, KAKENHI Grant Number 21K03204. The author thanks Professor Yuta Suzuki for telling him how to construct a smooth increasing function $f$ in  Lemma 4.5.  The author  sincerely thanks to the referee of this paper for reading the first manuscript very carefully and giving a lot of valuable comments and suggestions.


\noindent
Kanto Gakuin University, \\
Kanazawa, Yokohama\\
Kanagawa, Japan\\
E-mail address: sono@kanto-gakuin.ac.jp

\end{document}